\documentclass[12pt]{amsart}
\usepackage{graphicx, enumerate, amsmath, amssymb,amsthm,manfnt,hyperref,amscd,braket}
\usepackage[utf8]{inputenc} 
\allowdisplaybreaks

\newcommand{\R}{{\mathbb{R}}}
\newcommand{\C}{{\mathbb{C}}}

\newcommand{\Z}{\mathbb{Z}}
\newcommand{\Q}{\mathbb{Q}}

\newcommand{\tmop}[1]{\ensuremath{\operatorname{#1}}}

\renewcommand{\Im}{\tmop{Im}}

\theoremstyle{plain}
\newtheorem{theorem}[equation]{Theorem}
\newtheorem{proposition}[equation]{Proposition}
\newtheorem{lemma}[equation]{Lemma}
\newtheorem{corollary}[equation]{Corollary}
\newtheorem{definition}[equation]{Definition}
\theoremstyle{remark}
\newtheorem{remark}[equation]{Remark}
\numberwithin{equation}{section}
\def\rmd{\mathrm{d}}

\newcommand{\Bp}{\mathbf{B}}
\newcommand{\Fo}{\mathbf{F}}

\newcommand{\U}{\mathbb{U}}

\newcommand{\Ht}{\mathbb{H}}

\newcommand{\D}{\mathbb{D}}
\newcommand{\inv}{\mathbb{I}^{\tau}(\mathbb{D}^2)}

\newcommand{\Sb}{\mathbb{G}}

\title[Bergman Projection]{$L^p$ regularity of the Bergman Projection on domains covered by the polydisk}
\author{Liwei Chen}
\address[Liwei Chen]{The Ohio State University, Department of
Mathematics, Columbus, OH 43210}
\email{chenliwei@wustl.edu}

\author{Steven G. Krantz}
\address[Steven G. Krantz]{Washington University in St Louis, Department of Mathematics and Statistics, St Louis, MO 63130}
\email{sk@math.wustl.edu}

\author{Yuan Yuan}
\address[Yuan Yuan]{Syracuse University, Department of Mathematics, Syracuse, NY 13244}
\email{yyuan05@syr.edu}

\thanks{The third author is supported by National Science Foundation grant DMS-1412384, Simons Foundation grant (\#429722 Yuan Yuan) and CUSE grant program at Syracuse University}
\begin{document}

\maketitle
\begin{abstract}
If a bounded domain can be covered by the polydisk through a rational proper holomorphic map, then the Bergman projection is $L^p$-bounded for $p$ in a certain range depending on the ramified rational covering. This result can be applied to the symmetrized polydisk and to the Hartogs triangle with exponent $\gamma$.
\end{abstract}

\section{Introduction}
For a bounded domain $\Omega$ in $\C^n$, denote the Bergman space by $A^2(\Omega)=L^2(\Omega)\cap\mathcal{O}(\Omega)$. The Bergman projection is the orthogonal projection $\Bp_{\Omega}:L^2(\Omega)\to A^2(\Omega)$. 
The mapping properties of the Bergman projection on $L^p$ spaces have been studied for many years.

In the late 1970s and early 1980s, people considered smoothly bounded domains with various convexity conditions on the boundary, see for example \cite{PhoSte77,NagRosSteWai89,McNSte94,CharpentierDupain06}. To show the $L^p$-boundedness, the general recipe is to construct a quasi-distance and control the Bergman kernel in terms of the quasi-distance and its derivatives. Considering the Bergman projection as an integral operator, one can prove the $L^p$-boundedness for the Bergman projection for $1<p<\infty$. However, Barrett in \cite{Bar84,Bar92} discovered that there are smooth domains on which the Bergman projection behaves irregularly on $L^p$ spaces.

Later in the 21st century, people also discovered that the $L^p$-regularity of the Bergman projection has degenerate $p$ range, when considering non-smooth domains, see for example \cite{LanSte04,KraPel,Zey13,Chen14,ChaZey16,EdhMcN16,Huo18}. In particular, the boundary geometry of these non-smooth domains plays an essential role. While in \cite{LanSte04} Lanzani and Stein focus on simply connected planar domains and show that the $p$ ranges are certain intervals depending on the regularity of the boundary of the domain, it is a different story when one considers higher-dimensional, non-smooth domains---the $p$ range can even degenerate to the singleton $\{2\}$ (cf. \cite{Zey13,ChenZey,EdhMcN17}). 
What kind of geometry forces such a degeneracy of the $p$ range is still a mystery.

In this article, a certain class of domains in $\C^n$ is considered. Namely, a class of bounded domains that can be covered by the polydisk $\D^n$ through a rational proper holomorphic map. It is shown that these domains are of the first type: the $p$ range is always an interval with conjugate exponent endpoints (cf. Theorem \ref{maintheorem} in \S\ref{S:main}). It should be emphasized that the property of being covered by $\D^n$ through a 
rational proper holomorphic map is a geometric property of the domain, whereas $L^p$-regularity of the Bergman projection for a certain range of $p$ is an analytic property of the function spaces on the domain.

The idea of the proof is based on the Bergman projections transform in \cite{Bell81} and an application of the result in \cite{LanSte04}. The Bergman projection on the base domain is pulled back to the polydisk $\D^n$, and then is transferred to the product of upper half planes. From there, the $L^p$-regularity is reduced to a weighted integral inequality (see \eqref{weightedLponUn} in \S\ref{S:main}). By the basic facts of the class $A_p^+$ (see \S\ref{Apclass} for the definition of the class $A^+_p$), the weighted integral inequality is proved by showing that
the weight belongs to the class $A_p^+$. This powerful technique was first introduced by Lanzani and Stein in \cite{LanSte04} in one-variable. Their technique is applied to the higher dimensional case in this article. Here, the covering map being rational plays an important role. By the fundamental theorem of algebra and the factorization property (cf. Lemma \ref{ha}), it suffices to verify that each factor of the weight is in the class $A^+_p$ (see \S\ref{S:main} for details).

In the past 20 years, the symmetrized bidisk
\[
\Sb=\{(z_1+z_2,z_1z_2)\in\C^2\,|\,(z_1,z_2)\in\D\times\D\}
\]
has been studied intensively by the functional analysts (see for example \cite{AglYou00,AglYou04,AglLykYou18}). It is natural to ask what the Bergman theory on the symmetrized bidisk $\Sb$ is. Note that the symmetrized bidisk has the structure ``$z_1+z_2$", which crosses the two components of $\D^2$. So the Bergman theory on $\Sb$ cannot simply reduce to the ``one-variable" problem as on $\D^2$. However, we shall see 
in \S\ref{sympoly} that $\Sb$ can be covered by $\D^2$ through a rational proper holomorphic map. Indeed, symmetrized polydisk, the $n$-dimensional generalization of $\Sb$ is considered there. By employing the fundamental idea developed by Lanzani and Stein in \cite{LanSte04} and its generalization (cf. \S\ref{S:main}), the $L^p$ boundedness for the Bergman projection on the $n$-dimensional symmetrized polydisk is obtained. Moreover, as an example, under this ``covering mapping method", the largest possible interval for $p$ so that the Bergman projection is $L^p$-bounded has been computed for the symmetrized polydisks (cf. Theorem \ref{nsymdisk} in \S\ref{sympoly}).

Recently, Edholm and McNeal considered the Hartogs triangle 
\[
\Ht^{\gamma}=\{(z_1,z_2)\in\C^2\,:\,|z_1|^{\gamma}<|z_2|<1\}
\]
with exponent $\gamma\in\R^+$ in \cite{EdhMcN16, EdhMcN17}, where they call them ``fat Hartogs triangles''. 
It is shown in \S\ref{Hartogstriangle} that, when $\gamma$ is rational, $\Ht^{\gamma}$ can be covered by $\D\times\D^*$ through a rational proper holomorphic map, where $\D$ is the unit disk and $\D^*=\D\setminus\{0\}$. Since the Bergman spaces $A^2(\D)$ and $A^2(\D^*)$ are the same, our main result (Theorem \ref{maintheorem}) also applies. This is consistent with the result in \cite{EdhMcN16}. Edholm and McNeal gave a sharp range of $p$ there. On the other hand, when $\gamma$ is irrational, Edholm and McNeal showed in \cite{EdhMcN17} that the Bergman projection is $L^p$-bounded only if $p=2$. Combining this result with our main theorem, one can derive an interesting fact (Corollary \ref{nopropercover} in \S\ref{Hartogstriangle}) about the geometric mapping property of $\Ht^{\gamma}$---the Hartogs triangle with irrational exponent cannot be covered by $\D\times\D^*$ through a rational proper holomorphic mapping.

In addition to the $L^p$-regularity of the Bergman projection on $\Sb$, the mapping properties 
of the Friedrichs operator on $\Sb$ are also considered in this article (see \S\ref{MapProFriedrichs} for the definition of the Friedrichs operator and its relation with the Bergman projection). The Friedrichs operator is first introduced in \cite{Friedrichs}, and has been studied on planar domains in \cite{ShapiroUnbddQuad,ShapiroBook,PutinarShapiro1,PutinarShapiro2}. It is well-known that a planar domain is a quadrature domain if and only if its Friedrichs operator is of finite rank. In particular, if it is of rank one, then the domain is the unit disk $\D$ (cf. for example \cite{PutinarShapiro1}). Recently, the Friedrichs operator has been studied on higher dimensional domains. It is noticed that the Friedrichs operator possesses different types of smoothing properties (cf. \cite{HerMcN10,HerMcNStr,RavZey,ChenZey18}). In particular, Ravisankar and Zeytuncu consider some holomorphic extension properties
of the Friedrichs operator on higher dimensional domains with some rotational symmetry in \cite{RavZey}. Namely, every output function under the Friedrichs operator has a holomorphic extension on a larger domain. It is natural to ask whether it is because of the rotational symmetry of the domain that the Friedrichs operator possesses this smoothing property. However, the symmetrized bidisk is a counterexample to this 
question---it lacks rotational symmetry but its Friedrichs operator is of rank one (cf. Proposition \ref{notHartogs} and Theorem \ref{Friedrichsrank1} in \S\ref{MapProFriedrichs}). This suggests that a symmetric proper covering from $\D^2$ can probably do the job as well.

The article is organized as follows. In \S\ref{Apclass}, some basic facts about the $A^+_p$ class are proved. In \S\ref{S:main}, the main result is stated and is proved. The applications to symmetrized polydisks and to Hartogs triangles with exponent are considered in \S\ref{sympoly} and \S\ref{Hartogstriangle} respectively. The Bergman space on $\Sb$ and the corresponding Friedrichs operator are studied in \S\ref{Friedrichs}.


\section{Analysis of the Class $A_p^+$}
\label{Apclass}

Let $\U$ be the upper half plane and let $\rmd A$ denotes the standard Euclidean area measure in $\C$. 
\begin{definition} \rm
For $1<p<\infty$, a weight $\mu>0$ belongs to the class $A_p^+(\U)$ if there exists $C>0$, such that 
\begin{equation}\label{A}
N_D(\mu):=\left(\frac{1}{\pi R^2}\int_{D\cap \U} \mu(z) \rmd A(z)\right) \cdot \left(\frac{1}{\pi R^2}\int_{D\cap \U} \mu(z)^{-\frac{q}{p}}\rmd A(z) \right)^{\frac{p}{q}} \leq C
\end{equation} for any disk $D = D(x, R)= \{z \in \mathbb{C}: |z-x| < R, x\in \mathbb{R}\}$ centered at a point on the $x$-axis, where $\frac{1}{p} + \frac{1}{q} =1$.
\end{definition}

\begin{lemma}
\label{coeff}
Let $1 < p < \infty$.  For $w\in\C$, if $\mu\in A_p^+(\U)$, then $w\mu\in A_p^+(\U)$ with upper bound independent of $w$, i.e., $N_D(w \mu)$ is bounded from above by a uniform constant independent of $w$ and $D$.
\end{lemma}

\begin{proof}
The conclusion is trivial if $w=0$. Assume $w\neq0$. Since
\[
\frac{1}{\pi R^2}\int_{D\cap \U} w\mu(z) \rmd A(z)=\frac{w}{\pi R^2}\int_{D\cap \U} \mu(z) \rmd A(z)
\]
and
\[
\left(\frac{1}{\pi R^2}\int_{D\cap \U} \big(w\mu(z)\big)^{-\frac{q}{p}}\rmd A(z) \right)^{\frac{p}{q}}=w^{-1}\left(\frac{1}{\pi R^2}\int_{D\cap \U} \mu(z)^{-\frac{q}{p}}\rmd A(z) \right)^{\frac{p}{q}}
\]
for any $D = D(x, R)$ with $x\in\R$, \eqref{A} is verified by
\[
N_D(w\mu)=N_D(\mu)\le C.
\]
\end{proof}

\begin{lemma}\label{ha} \sl
Let $1 < p < \infty$.  If $\mu_j \in A^+_p(\U)$ for $j=1, 2$, then $\mu_1^\theta \mu_2^{1-\theta} \in A^+_p(\U)$ for any $\theta \in [0, 1]$.
\end{lemma}

\begin{proof}
There is nothing to prove if $\theta=0,1$. So assume $\theta\in(0,1)$. Let $r=1/\theta$. If $1/r+1/r'=1$, then $r'=1/(1-\theta)$. Let $q$ be the conjugate exponent of $p$. For any disk $D$ as in \eqref{A}, applying H\"{o}lder's inequality, one obtains
\begin{equation}
\label{eq1}
\begin{split}
\frac{1}{\pi R^2}\int_{D\cap \U} \mu_1^{\theta}(z)\mu_2^{1-\theta}(z) \rmd A(z)\le
& \left(\frac{1}{\pi R^2}\int_{D\cap \U} \mu_1(z) \rmd A(z)\right)^{1/r}\cdot\\
& \qquad\qquad\left(\frac{1}{\pi R^2}\int_{D\cap \U} \mu_2(z) \rmd A(z)\right)^{1/r'}
\end{split}
\end{equation}
and
\begin{equation}
\label{eq2}
\begin{split}
\left(\frac{1}{\pi R^2}\int_{D\cap \U} [\mu_1^{\theta}(z)\mu_2^{1-\theta}(z)]^{-\frac{q}{p}}\rmd A(z) \right)^{\frac{p}{q}}\le
& \left(\frac{1}{\pi R^2}\int_{D\cap \U} \mu_1(z)^{-\frac{q}{p}} \rmd A(z)\right)^{\frac{1}{r}\cdot\frac{p}{q}}\cdot\\
& \qquad\qquad\left(\frac{1}{\pi R^2}\int_{D\cap \U} \mu_2(z)^{-\frac{q}{p}} \rmd A(z)\right)^{\frac{1}{r'}\cdot\frac{p}{q}}.
\end{split}
\end{equation}
Since $\mu_1,\mu_2\in A^+_p(\U)$, $N_D(\mu_1),N_D(\mu_2)\le C$ for some $C>0$. Multiplying \eqref{eq1} and \eqref{eq2}, one obtains
\[
N_D(\mu_1^{\theta}\mu_2^{1-\theta})\le N_D(\mu_1)^{1/r}N_D(\mu_2)^{1/r'}\le C.
\]
Since $D$ is arbitrary, this completes the proof.
\end{proof}

\begin{proposition}\label{mu1}	     \sl
Let $\mu(z)=\left|z-w\right|^{\frac{\alpha(2-p)}{\theta}}$ be a weight on $\U$, where $\theta\in(0,1)$, $\alpha>0$, $p\in(1,\infty)$, and $w\in\C$. If $p \in \left(\frac{2\alpha+2\theta}{\alpha+2\theta}, \frac{2\alpha+2\theta}{\alpha}\right)$, then $\mu\in A^+_p(\U)$ with an upper bound independent of $w$, i.e. $N_D(\mu)$ is bounded from above by a uniform constant independent of $w$ and $D$.
\end{proposition}

\begin{proof}
The inequality \eqref{A} will be proved for different types of disks $D=D(x, R)$, where $x \in \mathbb{R}$ and $R>0$. Let $d(x, w) = L$.

Now assume that  $L \geq 10R$. If $z\in D$, then $ L-R \leq  |z-w| \leq L+R $. It follows from the definition that
\[
N_D(\mu)  \leq \left(\frac{1}{2}\right)^{\frac{p+q}{q}}\frac{\max_{z \in D}\mu(z)}{\min_{z \in D}\mu(z)} \leq \left(\frac{1}{2}\right)^{\frac{p+q}{q}}\left(\frac{L+R}{L-R}\right)^{\frac{\alpha|2-p|}{\theta}} \leq \left(\frac{1}{2}\right)^{\frac{p+q}{q}}\left( \frac{11}{9}\right)^{\frac{\alpha|2-p|}{\theta}}
\]
for any given $w\in\C$.

Assume that $L < 10R$. Let $D'=D(w, 20R)$. Then $D\subset D'$. Therefore
\[
\int_{D \cap \U} \mu(z) \rmd A(z) \leq \int_{D'}\mu(z) \rmd A(z)
\]
and 
\[
\int_{D \cap \U} \mu(z)^{-\frac{q}{p}} \rmd A(z) \leq \int_{D'}\mu(z)^{-\frac{q}{p}} \rmd A(z).
\]
On the other hand, 
\[
\int_{D'} \mu(z) \rmd A(z) = 2\pi \int^{20R}_0 r^{\frac{\alpha(2-p)}{\theta}} r dr = \frac{2\pi}{2+\frac{\alpha(2-p)}{\theta}} \left(20R \right)^{2+\frac{\alpha(2-p)}{\theta}}
\]
provided $2+\frac{\alpha(2-p)}{\theta} >0$ and  
\[
\int_{D'} \mu(z)^{-\frac{q}{p}}  \rmd A(z) = 2\pi \int^{20R}_0 r^{-\frac{\alpha(2-p)q}{p\theta}} r dr = \frac{2\pi}{2-\frac{\alpha(2-p)q}{p\theta}} \left(20R \right)^{2-\frac{\alpha(2-p)q}{p\theta}}
\]
provided $2-\frac{\alpha(2-p)q}{p\theta} >0$. Therefore, when $2+\frac{\alpha(2-p)}{\theta} >0$ and $2-\frac{\alpha(2-p)q}{p\theta} >0$, 
\begin{equation}\notag
\begin{split}
N_D(\mu) &\leq \frac{1}{(\pi R^2)^{1+\frac{p}{q}}}\left(\int_{D'}\mu(z) \rmd A(z) \right)\cdot \left( \int_{D'}\mu(z)^{-\frac{q}{p}} \rmd A(z)\right)^{\frac{p}{q}} \\
&\leq \frac{1}{\pi^\frac{p+q}{q}}\frac{2\pi}{2+\frac{\alpha(2-p)}{\theta}} \cdot\frac{2\pi}{2-\frac{\alpha(2-p)q}{p\theta}} \cdot 20^{2(1+\frac{p}{q}) } R^{2+\frac{\alpha(2-p)}{\theta} -2} R^{(2-\frac{\alpha(2-p)q}{p\theta}-2)\frac{p}{q}} \\
&=\frac{1}{\pi^\frac{p+q}{q}}\frac{2\pi}{2+\frac{\alpha(2-p)}{\theta}}\cdot \frac{2\pi}{2-\frac{\alpha(2-p)q}{p\theta}} \cdot 20^{2(1+\frac{p}{q}) }.
\end{split}
\end{equation}
Combining $2+\frac{\alpha(2-p)}{\theta} >0$ and $2-\frac{\alpha(2-p)q}{p\theta} >0$ with $\frac{1}{p} + \frac{1}{q} =1$, one sees that $p \in \left(\frac{2\alpha+2\theta}{\alpha+2\theta}, \frac{2\alpha+2\theta}{\alpha}\right)$. This completes the proof.
\end{proof}

\begin{proposition}\label{mu2} \sl
Let $\mu(z)=\left|z-w\right|^{\frac{-\beta(2-p)}{\sigma}}$ be a weight on $\U$, where $\sigma\in(0,1)$, $\beta>2\sigma$, $p\in(1,\infty)$, and $w\in\C$. If $p \in \left(\frac{2\beta-2\sigma}{\beta}, \frac{2\beta-2\sigma}{\beta-2\sigma}\right)$, then $\mu\in A^+_p(\U)$ with a bound independent of $w$.
\end{proposition}

\begin{proof}
By a similar argument as in the proof of Proposition \ref{mu1}, one can prove that $\mu \in A^+_p(\U)$ with a bound independent of $w$ if $2+\frac{-\beta(2-p)}{\sigma} >0$ and $2-\frac{-\beta(2-p)q}{p\sigma}>0$. When $\beta>2\sigma>0$, one can derive $p \in \left(\frac{2\beta-2\sigma}{\beta}, \frac{2\beta-2\sigma}{\beta-2\sigma}\right)$.
\end{proof}

For the later application, we state the result \cite[Proposition 4.5]{LanSte04} at the end of this section. For a proof, see \cite[\S 4]{LanSte04} for details.

\begin{theorem}[Lanzani-Stein 04]
\label{LS04prop4.5}
Suppose that $1<p<\infty$. Let $\Bp_{\U}$ be the Bergman projection on $\U$ and $\mu$ be a weight on $\U$. Then $\Bp_{\U}$ is bounded on $L^p(\U,\mu)$ if and only if $\mu\in A^+_p(\U)$. Here $L^p(\U,\mu)$ is the 
space consisting all measurable functions $f$ on $\U$ such that
\[
\|f\|_{L^p(\U,\mu)}:=\left(\int_{\U}|f(z)|^p\mu(z)\rmd A(z)\right)^{1/p}<\infty.
\]
\end{theorem}


\section{Main Theorem}
\label{S:main}

Let $\D^n\subset\C^n$ be the polydisk and let $\Omega\subset\C^n$ be a bounded domain. Assume that $\Phi: \D^n\to\Omega$ is a surjective proper rational holomorphic mapping. Then $\Phi$ is a ramified covering map of finite order and each component of $\Phi$ is a rational function whose denominator is nonzero. We will show that the $p$-range for the $L^p$-boundedness of the Bergman projection $\Bp_{\Omega}$ never degenerates to just $p=2$.

\begin{theorem}
\label{maintheorem} \sl
The Bergman projection $\Bp_{\Omega}$ on $\Omega$ is $L^p(\Omega)$-bounded for $p\in(r,r')$, where $r<2$ and $r'>2$ are two conjugate exponents depending on the ramified rational covering.
\end{theorem}

\begin{proof}
By \cite[Theorem 1]{Bell81}, the Bergman projections transform in the following form
\[
\Bp_{\D^n}(J_{\C}\Phi\cdot(h\circ\Phi))=J_{\C}\Phi\cdot(\Bp_{\Omega}(h)\circ\Phi)\qquad\text{for\,\,\,} h\in L^2(\Omega),
\]
where $J_{\C}\Phi$ is the complex Jacobian determinant of $\Phi$. So, to prove the $L^p$-estimate of $\Bp_{\Omega}$,
\[
\|\Bp_{\Omega}(h)\|_{L^p(\Omega)}\le C\|h\|_{L^p(\Omega)}\qquad\text{for\,\,\,} h\in L^2(\Omega)\cap L^p(\Omega),
\]
it is equivalent to show that
\begin{equation}
\label{pullbackonD}
\int_{\D^n}\left|\Bp_{\D^n}\Big(J_{\C}\Phi\cdot\big(h\circ\Phi\big)\Big)\cdot\Big(J_{\C}\Phi\Big)^{-1}\right|^p\cdot\left|J_{\C}\Phi\right|^2\rmd V\le C_p\int_{\D^n}|h\circ\Phi|^p|J_{\C}\Phi|^2\rmd V,
\end{equation}
where $\rmd V$ is the standard Euclidean volume measure. Let $g=J_{\C}\Phi\cdot(h\circ\Phi)$. To prove \eqref{pullbackonD}, it suffices to show that
\begin{equation}
\label{weightedLponDn}
\int_{\D^n}\left|\Bp_{\D^n}(g)\right|^p\cdot\left|J_{\C}\Phi\right|^{2-p}\rmd V\le C_p\int_{\D^n}|g|^p\cdot|J_{\C}\Phi|^{2-p}\rmd V
\end{equation}
for $g\in L^p(\D^n,|J_{\C}\Phi|^{2-p})$, the $L^p$ space on $\D^n$ with weight $|J_{\C}\Phi|^{2-p}$.

Consider the Cayley transform $\psi:\U\to\D$ given by
\[
\psi(z)=\frac{i-z}{i+z},
\]
where $z\in\U=\{z\in\C\,:\,\Im(z)>0\}$. Let $\Psi=\otimes_{j=1}^n\psi:\U^n\to\D^n$ be the biholomorphism. 
Apply the Bergman projections transform \cite[Theorem 1]{Bell81} to $\Psi:\U^n\to\D^n$ and pull back from $\D^n$ to $\U^n$ as in \eqref{pullbackonD}. Let $f=J_{\C}\Psi\cdot(g\circ\Psi)$. To prove \eqref{weightedLponDn}, it suffices to show
that
\begin{equation}
\label{weightedLponUn}
\int_{\U^n}\left|\Bp_{\U^n}(f)\right|^p\cdot\left|Q\right|^{2-p}\rmd V\le C_p\int_{\U^n}|f|^p\cdot|Q|^{2-p}\rmd V,
\end{equation}
for $f\in L^p(\U^n,|Q|^{2-p})$, where $Q=J_{\C}\Psi\cdot((J_{\C}\Phi)\circ\Psi)$.

Note that $\Bp_{\U^n}=\otimes_{j=1}^n\Bp_{\U}$. Repeatly apply Theorem \ref{LS04prop4.5} $n$ times. To prove \eqref{weightedLponUn}, it suffices to check:
\begin{enumerate}
\item $|Q|^{2-p}$ as a weight in the variable $z_1$ is in $A_p^+(\U)$ with a uniform bound independent of $z_2,\dots,z_n$;
\item $|Q|^{2-p}$ as a weight in the variable $z_2$ is in $A_p^+(\U)$ with a uniform bound independent of $z_1,z_3,\dots,z_n$;
\[
\dots
\]
\item[(n)] $|Q|^{2-p}$ as a weight in the variable $z_n$ is in $A_p^+(\U)$ with a uniform bound independent of $z_1,\dots,z_{n-1}$.
\end{enumerate}
Without loss of generality, it suffices to check (1) above. Namely, for a.e. $z_2,\dots,z_n$, $|Q(\cdot,z_2,\dots,z_n)|^{2-p}\in A_p^+(\U)$ with a uniform bound $C$ independent of $z_2,\dots,z_n$.

Since $\Phi$ and $\Psi$ are rational, so is $Q$. Let $Q(z)=\frac{P_1(z_1,\dots,z_n)}{P_2(z_1,\dots,z_n)}$, where $P_1$ and $P_2$ are polynomials in $z_1,\dots,z_n$. For a.e. $z_2,\dots,z_n\in\U$, consider $P_1$ and $P_2$ as polynomials in $z_1$. By the fundamental theorem of algebra, these polynomials can be written as
\[
P_1(z)=a_0(z_1-a_1)^{\alpha_1}\dots(z_1-a_k)^{\alpha_k}\qquad\text{and}\qquad P_2(z)=b_0(z_1-b_1)^{\beta_1}\dots(z_1-b_l)^{\beta_l},
\]
where $\alpha_1,\dots,\alpha_k,\beta_1,\dots,\beta_l\in\Z^+$ and $a_0,\dots,a_k$, $b_0\dots,b_l$ depend on $z_2,\dots,z_n$ but are independent of $z_1$.

Since $a_0/b_0$ is independent of $z_1$, by Lemma \ref{coeff} and \ref{ha}, it suffices to assume $a_0/b_0=1$ and check
\begin{equation}
\label{holderweight}
|z_1-a_1|^{\frac{\alpha_1(2-p)}{\theta_1}},\dots,|z_1-a_k|^{\frac{\alpha_k(2-p)}{\theta_k}},|z_1-b_1|^{\frac{-\beta_1(2-p)}{\sigma_1}},\dots,|z_1-b_l|^{\frac{-\beta_l(2-p)}{\sigma_l}}\in A^+_p(\U)
\end{equation}
for some $\theta_1,\dots,\theta_k,\sigma_1,\dots,\sigma_l\in(0,1)$ independent of $z_2,\dots,z_n$ with $(\theta_1+\dots+\theta_k)+(\sigma_1+\dots+\sigma_l)=1$. Since $\beta_1,\dots,\beta_l\in\Z^+$, take $\sigma_1,\dots,\sigma_l\in(0,1/2)$. 
By Propositions \ref{mu1} and \ref{mu2}, the condition \eqref{holderweight} holds when
\[
p\in\left(\bigcap_{j=1}^{k}\left(\frac{2\alpha_j+2\theta_j}{\alpha_j+2\theta_j},\frac{2\alpha_j+2\theta_j}{\alpha_j}\right)\right)\bigcap\left(\bigcap_{s=1}^{l}\left(\frac{2\beta_s-2\sigma_s}{\beta_s},\frac{2\beta_s-2\sigma_s}{\beta_s-2\sigma_s}\right)\right)=:I_1.
\]
Note that each interval above contains $2$ and its endpoints are conjugate exponents. Hence $I_1$ is nonempty and write $I_1=(r_1,r'_1)$, where $r_1<2$ and $r'_1>2$ are conjugate exponents.

In a similar fashion, conditions (2)--(n) hold when $p\in I_2,\dots,p\in I_n$, respectively. Here for each $j=2,\dots,n$, $I_j=(r_j,r'_j)$ where $r_j<2$ and $r'_j>2$ are conjugate exponents. Write $I=\cap_{j=1}^nI_j= (r,r')$, where $r<2$ and $r'>2$ are conjugate exponents. Therefore, $\Bp_{\Omega}$ is $L^p(\Omega)$-bounded for $p\in I$.
\end{proof}


\section{Application to Symmetrized Polydisks}
\label{sympoly}

For $w=(w_1,w_2,\dots,w_n)\in\C^n$, we denote the symmetric polynomials by
\begin{align*}
& p_1(w)=\sum_{j=1}^n w_j,\\
& p_2(w)=\sum_{j<k}w_jw_k,\\
& p_3(w)=\sum_{j<k<l}w_jw_kw_l,\\
& \cdots \\
& p_n(w)=w_1w_2\cdots w_n.
\end{align*}

\begin{definition}
The $n$-dimensional symmetrized polydisk is defined by
\[
\Sb^n=\{z=(p_1(w),p_2(w),\dots,p_n(w))\in\C^n\,:\,w\in\D^n\}.
\]
\end{definition}

\begin{proposition}
\label{coveringmapn!}
Let $\Phi_n:\D^n\to\Sb^n$ be the holomorphic mapping defined by
\[
\Phi_n(w)=(p_1(w),p_2(w),\dots,p_n(w)).
\]
Then $\Phi_n$ is a ramified rational proper covering map of order $n!$ with complex Jacobian determinant
\begin{equation}
\label{njacobian}
J_{\C}\Phi_n(w)=\prod_{j<k}(w_j-w_k).
\end{equation}
\end{proposition}

\begin{proof}
Since $p_1,\dots,p_n$ are polynomials, $\Phi_n$ is rational and proper. Note that $\Sb^n=\Phi_n(\D^n)$. As a proper holomorphic surjective mapping, $\Phi_n:\D^n\to\Sb^n$ is a ramified covering. If $\tau_n$ is a permutation on $\{1,2,\dots,n\}$, then
\[
\Phi_n(w_1,\dots,w_n)=\Phi_n(w_{\tau_n(1)},\dots,w_{\tau_n(n)}).
\]
So $\Phi_n$ is of order $n!$.

Next, we prove \eqref{njacobian} by induction on $n$. When $n=1$, \eqref{njacobian} is trivially
\[
J_{\C}\Phi_1(w)=1.
\]
Assume that \eqref{njacobian} holds for $n=m$. We show \eqref{njacobian} holds for $n=m+1$ as well. Note that if $w_j=w_k$ for any $1\le j<k\le m+1$, then $J_{\C}\Phi_{m+1}(w)=0$. So $J_{\C}\Phi_{m+1}(w)$ is divisible by $\prod_{j<k}(w_j-w_k)$. On the other hand, for $j=1,\dots,m+1$ the function $J_{\C}\Phi_n$ is a polynomial in $w_j$ with leading power $m$, which is the same as $\prod_{j<k}(w_j-w_k)$. So
\begin{equation}
\label{m+1jacobian}
J_{\C}\Phi_{m+1}(w)=c\prod_{1\le j<k\le m+1}(w_j-w_k)
\end{equation}
for some constant $c\neq0$.

In \eqref{m+1jacobian}, let $w_{m+1}=0$. The last row of the determinant on the lefthand side of \eqref{m+1jacobian} becomes $(0,\dots,0,w_1\cdots w_m)$. Expanding this row from the determinant gives $w_1\cdots w_mJ_{\C}\Phi_m(w)$. On the other hand, the righthand side of \eqref{m+1jacobian} becomes $cw_1\cdots w_m\prod_{1\le j<k\le m}(w_j-w_k)$. Therefore, \eqref{m+1jacobian} becomes
\[
J_{\C}\Phi_{m}(w)=c\prod_{1\le j<k\le m}(w_j-w_k).
\]
By the inductive hypothesis, $c=1$. So \eqref{njacobian} holds for $n=m+1$. This completes the proof.
\end{proof}

As in the proof of Theorem \ref{maintheorem}, let $\Psi=\otimes_{j=1}^n\psi:\U^n\to\D^n$ be a biholomorphism, where $\psi:\U\to\D$ is the Cayley transform
\[
\psi(z)=\frac{i-z}{i+z}.
\]
Then the Bergman projection $\Bp_{\Sb^n}$ on $\Sb^n$ is $L^p(\Sb^n)$-bounded if \eqref{weightedLponUn} holds with
\[
Q(z)=J_{\C}\Psi(z)\cdot((J_{\C}\Phi_n)\circ\Psi(z))=\frac{c\prod_{j<k}(z_j-z_k)}{\prod_{j=1}^n(i+z_j)^{n+1}}
\]
for some universal constant $c$.

Since $Q$ is symmetric in $z_1,\dots,z_n$, it suffices to check any of conditions (1)-(n) in the proof of 
Theorem \ref{maintheorem}. Without loss of generality, we check (1). As in \eqref{holderweight}, it 
suffices to check that 
\begin{equation}
\label{symdiskholderweight}
|i+z_1|^{\frac{-(n+1)(2-p)}{\theta_1}},|z_1-z_2|^{\frac{2-p}{\theta_2}}, \dots , |z_1-z_n|^{\frac{2-p}{\theta_n}}\in A^+_p(\U)
\end{equation}
with a bound independent of $z_2,\dots,z_n$ for some $\theta_1,\dots,\theta_n\in(0,1)$ with $\theta_1+\cdots+\theta_n=1$.

By Propositions \ref{mu1} and \ref{mu2}, the condition \eqref{symdiskholderweight} holds when
\begin{equation}
\label{ninterval}
p\in\left(\frac{2(n+1)-2\theta_1}{n+1},\frac{2(n+1)-2\theta_1}{n+1-2\theta_1}\right)\bigcap\left(\frac{2+2\theta_2}{1+2\theta_2},2+2\theta_2\right)\bigcap\cdots\bigcap\left(\frac{2+2\theta_n}{1+2\theta_n},2+2\theta_n\right).
\end{equation}
Note that the last $n-1$ intervals are symmetric in  $\theta_2,\dots,\theta_n$. Given $\theta_1\in(0,1)$, the largest possible intersection of these $n-1$ intervals occurs when $\theta:=\theta_2=\cdots=\theta_n$. So \eqref{ninterval} becomes
\begin{equation}
\label{2interval}
p\in\left(\frac{2+2\theta}{1+2\theta},2+2\theta\right)\bigcap\left(\frac{2n+2(n-1)\theta}{n+1},\frac{2n+2(n-1)\theta}{n-1+2(n-1)\theta}\right),
\end{equation}
since $\theta_1+(n-1)\theta=1$. As $\theta$ varies from $0$ to $1$, in \eqref{2interval} the first interval is expanding while the second interval is shrinking. Since the endpoints are conjugate exponents, the largest possible intersection occurs when the two intervals are identical. This is achieved by setting
\[
\theta=\frac{\sqrt{n^2-1}-n+1}{2n-2}
\]
and \eqref{2interval} becomes
\begin{equation}
\label{prangesymdisk}
p\in\left(\frac{\sqrt{n^2-1}+n-1}{\sqrt{n^2-1}},\frac{\sqrt{n^2-1}+n-1}{n-1}\right).
\end{equation}

We summarize what we have proved in the following.

\begin{theorem}
\label{nsymdisk} \sl
The Bergman projection $\Bp_{\Sb^n}$ on the $n$-dimensional symmetrized polydisk $\Sb^n$ is $L^p(\Sb^n)$-bounded if \eqref{prangesymdisk} holds.
\end{theorem}

In particular, when $n=2$, the classical symmetrized bidisk
\[
\Sb:=\Sb^2=\{(z_1+z_2,z_1z_2)\in\C^2\,|\,(z_1,z_2)\in\D\times\D\}
\]
is of particular interest in the geometric function theory (cf. \cite{AglLykYou18, AglYou00, AglYou04}).

\begin{corollary}
The Bergman projection $\Bp_\Sb$ is $L^p(\Sb)$-bounded for $p \in \left(\frac{\sqrt{3}+1}{\sqrt{3}},\sqrt{3}+1 \right)$.
\end{corollary}


\section{Application to Hartogs Triangles}
\label{Hartogstriangle}

For $\gamma\in\R^+$, let
\[
\Ht^{\gamma}=\{(z_1,z_2)\in\C^2\,:\,|z_1|^{\gamma}<|z_2|<1\}
\]
be the Hartogs triangle with exponent $\gamma$. Since the Bergman space $A^2(\D)$ is the same as $A^2(\D^*)$, the result in \S\ref{S:main} applies to any domain $\Omega\subset\C^2$ with a rational proper covering mapping $\Phi:\D\times\D^*\to\Omega$.

When $\gamma\in\Q^+$, let $\gamma=\frac{m}{n}$ for some $m,n\in\Z^+$ with $\text{gcd}(m,n)=1$. The holomorphic mapping $\Phi:\D\times\D^*\to\Ht^{m/n}$ given by
\[
\Phi(w_1,w_2)=(w_1w_2^n,w_2^m)\qquad\text{for\,\,\,\,\,}(w_1,w_2)\in\D\times\D^*
\]
is a rational proper covering map.
\begin{corollary}   \sl
The Bergman projection $\Bp_{\Ht^{m/n}}$ is $L^p(\Ht^{m/n})$-bounded for $p\in(r,r')$, where $r<2$ and $r'>2$ are conjugate exponents.
\end{corollary}

\begin{remark} \rm
Edholm and McNeal obtained in \cite{EdhMcN16} the precise nondegenerate interval of $p$ for which the Bergman projection is $L^p$-bounded. 
\end{remark}

When $\gamma$ is irrational, Edholm and McNeal showed that the Bergman projection $\Bp_{\Ht^{\gamma}}$ is $L^p(\Ht^{\gamma})$-bounded only when $p=2$. Their result together with Theorem \ref{maintheorem} implies the following result.

\begin{corollary} \rm
\label{nopropercover}
There is no rational proper covering map from $\D\times\D^*$ to $\Ht^{\gamma}$ when $\gamma$ is irrational.
\end{corollary}

\begin{remark} \rm
This geometric property of $\Ht^{\gamma}$ is obtained by an analytic method. Namely, the geometric mapping properties of Hartogs triangles with rational and irrational exponent are significantly different.
\end{remark}

This idea can be applied to higher dimensional domains as well.

\begin{corollary}   \sl
For any bounded domain $\Omega\subset\C^n$, if its Bergman projection $\Bp_{\Omega}$ is $L^p(\Omega)$-bounded only when $p=2$, then it cannot be covered by $\D^n$ through a rational proper holomorphic map.
\end{corollary}

\begin{remark}	\rm
There are examples in $\C^2$ in \cite{ChenZey, Zey13} other than $\Ht^{\gamma}$ mentioned above.
\end{remark}


\section{The Friedrichs operator on $\Sb$}
\label{Friedrichs}

\subsection{The Pull-Back Bergman Space}
\label{branchcover}
Let
\[
\Sb:=\Sb^2=\{(z_1+z_2,z_1z_2)\in\C^2\,|\,(z_1,z_2)\in\D\times\D\}
\]
be the symmetrized bidisk in $\C^2$. By Proposition \ref{coveringmapn!},
\[
\Phi: \D^2\to\Sb
\]
where $\Phi(z):=\Phi_2(z_1,z_2)=(z_1+z_2,z_1z_2)$ is a rational proper covering map of order $2$ with Jacobian determinant $J_{\C}\Phi(z)=z_1-z_2$.

Define a symmetrization map on $\D^2$ by
\[
\tau:\D^2\to\D^2
\]
with $\tau(z_1,z_2)=(z_2,z_1)$. A measurable function $f$ on $\D^2$ is called $\tau$-\emph{invariant} if $f\circ\tau=f$. We denote the set of $\tau$-invariant functions by $\inv$.
If $g\in A^2(\Sb)$, then $g\circ\Phi=g\circ\Phi\circ\tau$, so $g\circ\Phi\in\inv$. Let $\nu(z)=|z_1-z_2|^2$ on $\D^2$ and let $A^2(\D^2,\nu)=L^2(\D^2,\nu)\cap\mathcal{O}(\D^2)$ be the weighted Bergman space with norm
\[
\|f\|_{A^2(\D^2,\nu)}:=\left(\int_{\D^2}|f(z)|^2\nu(z)\,\rmd V(z)\right)^{1/2},
\]
where $\rmd V(z)$ is the standard Euclidean volume measure in $z$. By change of variables, if $g\in A^2(\Sb)$, then $g\circ\Phi\in A^2(\D^2,\nu)$ and $\|g\|^2_{L^2(\Sb)} = \frac{1}{2}\|g\circ\Phi\|^2_{L^2(\D^2, \nu)}$. 
On the other hand, if $f\in A^2(\D^2,\nu)\cap\inv$, then $f\circ\Phi^{-1}$ is a well-defined holomorphic function on $\Sb$ 
and thus $f\circ\Phi^{-1}\in A^2(\Sb)$ since
\[
\int_{\Sb}|f\circ\Phi^{-1}|^2\,dV=\frac{1}{2}\int_{\D^2}|f(z)|^2\nu(z)\,dV(z)<\infty.
\]
Therefore there is a $1$-$1$ correspondence between $A^2(\Sb)$ and $A^2(\D^2,\nu)\cap\inv$ through $\Phi$.

For each $f\in A^2(\D^2,\nu)\cap\inv$, let $h(z)=f(z)\cdot(z_1-z_2)$. Then $h\in A^2(\D^2)$ and $h\circ\tau=-h$. Write
\[
h(z_1,z_2)=\sum_{j,k\ge 0}c_{j,k}z^j_1z_2^k.
\] 
It follows from $h\circ\tau=-h$ that $c_{k,j}=-c_{j,k}$ for all $j,k\ge0$. Therefore, $h$ can be written as
\[
h(z_1,z_2)=\sum_{j>k}c_{j,k}\left(z_1^jz_2^k-z_1^kz_2^j\right).
\]
A direct computation shows that
\[
\int_{\D^2}|z_1^jz_2^k-z_1^kz_2^j|^2\,dV(z)=\frac{2\pi^2}{(k+1)(j+1)}.
\]
Hence $\left\{\sqrt{\frac{(k+1)(j+1)}{2\pi^2}}(z_1^jz_2^k-z_1^kz_2^j)\right\}_{j>k}$ is an orthonormal basis for the space $\left\{h \in A^2(\D^2) | h \circ \tau=-h\right\}$, and therefore $\left\{\sqrt{\frac{(k+1)(j+1)}{2\pi^2}}\frac{z_1^jz_2^k-z_1^kz_2^j}{z_1-z_2}\right\}_{j>k}$ is an orthonormal basis for $A^2(\D^2,\nu)\cap\inv$. Another computation shows that the Bergman kernel of the space $A^2(\D^2,\nu)\cap\inv$ is
\begin{equation*}
\begin{split}
B_{\nu}(z_1,z_2,\zeta_1,\zeta_2)
&=\sum_{j>k}\frac{(k+1)(j+1)}{2\pi^2}\cdot\frac{z_1^jz_2^k-z_1^kz_2^j}{z_1-z_2}\cdot\frac{\bar \zeta_1^j\bar\zeta_2^k-\bar\zeta_1^k\bar\zeta_2^j}{\bar\zeta_1-\bar\zeta_2}\\
&=\frac{1}{4\pi^2}\cdot\frac{1}{z_1-z_2}\cdot\frac{1}{\bar\zeta_1-\bar\zeta_2}\sum_{j\neq k}(k+1)(j+1)(z_1^jz_2^k-z_1^kz_2^j)(\bar \zeta_1^j\bar\zeta_2^k-\bar\zeta_1^k\bar\zeta_2^j)\\
&=\frac{1}{4\pi^2}\cdot\frac{1}{z_1-z_2}\cdot\frac{1}{\bar\zeta_1-\bar\zeta_2}\sum_{j,k}(k+1)(j+1)\Big[(z_1\bar\zeta_1)^j(z_2\bar\zeta_2)^k-(z_1\bar\zeta_2)^j(z_2\bar\zeta_1)^k\\
&\qquad\qquad\qquad\qquad\qquad\qquad\qquad -(z_1\bar\zeta_2)^k(z_2\bar\zeta_1)^j+(z_1\bar\zeta_1)^k(z_2\bar\zeta_2)^j \Big]\\
&=\frac{1}{2\pi^2}\cdot\frac{1}{(z_1-z_2)(\bar\zeta_1-\bar\zeta_2)}\left[\frac{1}{(1-z_1\bar\zeta_1)^2(1-z_2\bar\zeta_2)^2}-\frac{1}{(1-z_1\bar\zeta_2)^2(1-z_2\bar\zeta_1)^2}\right].
\end{split}
\end{equation*}

\begin{remark} \rm
Using Bell's result \cite{Bell82c}, one can also obtain  the Bergman kernel of the space $A^2(\D^2,\nu)\cap\inv$. But we will need the computation of the orthonormal basis later.
\end{remark}

\subsection{Mapping Properties of the Friedrichs Operator}
\label{MapProFriedrichs}

Let
\[
\Fo_{\Sb}:A^2(\Sb)\to A^2(\Sb)
\]
be the Friedrichs operator on $\Sb$ defined by $\Fo_{\Sb}(g)=\Bp_{\Sb}(\bar g)$ for $g\in A^2(\Sb)$, where $\bar g$ is the complex conjugate of $g$.

\begin{proposition}
\label{notHartogs}
The symmetrized bidisk is not a Hartogs domain in $\C^2$.
\end{proposition}

\begin{proof}
Let $z=(z_1,z_2)\in\Sb$ with $z_1=w_1+w_2$ and $z_2=w_1w_2$ for $w_1,w_2\in\D$. Then $|z_1|\le2$ and $(\pm2,1)\in\overline{\Sb}$ by taking $w_1=w_2=\pm1$. If $\Sb$ is circular in $z_1$, the only possibility is that $z_1$ is symmetric about $0$. If it is the case, then rotating $(2,1)$ in the $z_1$-direction counterclockwise by $\pi/2$ implies $(2i,1)\in\overline{\Sb}$. This is a contradiction, since $w_1+w_2=2i$ forces $w_1=w_2=i$ which gives $z_2=w_1w_2=-1$.

On the other hand, $(0,-1),(0,1)\in\overline{\Sb}$ by taking $w_1=-w_2=1$ and $w_1=-w_2=i$. Also, $|z_2|\le1$. If $\Sb$ is circular in $z_2$, the only possibility is that $z_2$ is symmetric about $0$. If it is the case, then rotating $(2,1)$ in the $z_2$-direction counterclockwisely by $\pi/2$ implies $(2,i)\in\overline{\Sb}$. This is a contradiction, since $w_1+w_2=2$ forces $w_1=w_2=1$ which gives $z_2=w_1w_2=1$. So $\Sb$ is not circular in $z_1$ nor $z_2$. This completes the proof.
\end{proof}

By Proposition \ref{notHartogs}, $\Sb$ is not a Reinhardt domain, nor a Hartogs domain. However, the Friedrichs operator $\Fo_{\Sb}$ is of rank one even if there is a lack of rotational symmetries on $\Sb$.

\begin{theorem}
\label{Friedrichsrank1}
The Friedrichs operator $\Fo_{\Sb}$ on $\Sb$ is of rank one. Moreover, 
\[
\Fo_{\Sb}:A^2(\Sb)\to H^{\infty}(\Sb):=L^{\infty}(\Sb)\cap\mathcal{O}(\Sb)
\]
and there exists $C>0$ such that
$\|\Fo_{\Sb}(g)\|_{L^{\infty}}\le C\|g\|_{L^2}$ for any $g\in A^2(\Sb)$. 
\end{theorem}

\begin{proof}
By the $1$-$1$ correspondence between $A^2(\Sb)$ and $A^2(\D^2,\nu)\cap\inv$, it suffices to look at the Friedrichs operator $\Fo_{\nu}$ on $A^2(\D^2,\nu)\cap\inv$. In \S \ref{branchcover}, it is shown that $\left\{\sqrt{\frac{(k+1)(j+1)}{2\pi^2}}\frac{z_1^jz_2^k-z_1^kz_2^j}{z_1-z_2}\right\}_{j>k}$ is an orthonormal basis for the Bergman space $A^2(\D^2,\nu)\cap\inv$. So, for $f\in A^2(\D^2,\nu)\cap\inv$, write
\[
f(z)=\sum_{m>n}a_{m,n}\frac{z_1^mz_2^n-z_1^nz_2^m}{z_1-z_2}.
\]
Note that the Bergman kernel of $A^2(\D^2,\nu)\cap\inv$ has the form
\[
B_{\nu}(z_1,z_2,\zeta_1,\zeta_2)=\sum_{j>k}\frac{(k+1)(j+1)}{2\pi^2}\cdot\frac{z_1^jz_2^k-z_1^kz_2^j}{z_1-z_2}\cdot\frac{\bar \zeta_1^j\bar\zeta_2^k-\bar\zeta_1^k\bar\zeta_2^j}{\bar\zeta_1-\bar\zeta_2}.
\]
So, by definition,
\begin{equation}
\label{expandF}
\begin{split}
\Fo_{\nu}(f)(z)
&=\int_{\D^2}B_{\nu}(z,\zeta)\overline{f(\zeta)}\nu(\zeta)\,\rmd V(\zeta)\\
&=\int_{\D^2}\sum_{j>k}\frac{(k+1)(j+1)}{2\pi^2}\cdot\frac{z_1^jz_2^k-z_1^kz_2^j}{z_1-z_2}\cdot\frac{\bar \zeta_1^j\bar\zeta_2^k-\bar\zeta_1^k\bar\zeta_2^j}{\bar\zeta_1-\bar\zeta_2}\sum_{m>n}\bar a_{m,n}\frac{\bar\zeta_1^m\bar\zeta_2^n-\bar\zeta_1^n\bar\zeta_2^m}{\bar\zeta_1-\bar\zeta_2}\nu(\zeta)\,\rmd V(\zeta)\\
&=\sum_{j>k,m>n}\frac{(k+1)(j+1)\bar a_{m,n}}{2\pi^2}\cdot\frac{z_1^jz_2^k-z_1^kz_2^j}{z_1-z_2}\cdot\\
&\qquad\qquad\qquad\qquad\int_{\D^2}(\bar\zeta_1\bar\zeta_2)^{k+n}\cdot\frac{(\bar\zeta_1^{j-k}-\bar\zeta_2^{j-k})(\bar\zeta_1^{m-n}-\bar\zeta_2^{m-n})}{(\bar\zeta_1-\bar\zeta_2)(\bar\zeta_1-\bar\zeta_2)}\nu(\zeta)\,\rmd V(\zeta).
\end{split}
\end{equation}
Note that, if $j>k$ and $m>n$, the term
\[
\frac{(\bar\zeta_1^{j-k}-\bar\zeta_2^{j-k})(\bar\zeta_1^{m-n}-\bar\zeta_2^{m-n})}{(\bar\zeta_1-\bar\zeta_2)(\bar\zeta_1-\bar\zeta_2)}
\]
in the integrand must be a polynomial in $\bar\zeta_1$ and $\bar\zeta_2$. On the other hand, expand the norm square of the weight function
\[
\nu(\zeta)=|\zeta_1-\zeta_2|^2=|\zeta_1|^2-\bar\zeta_1\zeta_2-\zeta_1\bar\zeta_2+|\zeta_2|^2.
\]
Therefore, by rotational symmetry of integration on $\D^2$, the only surviving term in \eqref{expandF} will be $k=n=0$ and $j=m=1$. Therefore
\begin{equation*}
\Fo_{\nu}(f)(z)=\frac{\bar a_{1,0}}{\pi^2}\int_{\D^2}\left( |\zeta_1|^2+|\zeta_2|^2\right)dV(\zeta)=\bar a_{1,0}.
\end{equation*}
This shows that $\Fo_{\nu}$ is of rank one. Hence the same property holds for $\Fo_{\Sb}$.

Note that $\bar a_{1,0}=\overline{f(0)}$. By the $1$-$1$ correspondence between $A^2(\Sb)$ and $A^2(\D^2,\nu)\cap\inv$,
\begin{equation}
\label{rankone}
\Fo_{\Sb}(g)(z)=\overline{g(0)}
\end{equation}
for $g\in A^2(\Sb)$. By definition, it is always true that $\|\Fo_{\Sb}(g)\|_{L^2}\le\|g\|_{L^2}$. So \eqref{rankone} gives
\[
\left(\int_{\Sb}|\overline{g(0)}|^2\,dV(\zeta)\right)^{1/2}=\|\Fo_{\Sb}(g)\|_{L^2}\le\|g\|_{L^2},
\]
which implies
\[
\|\Fo_{\Sb}(g)\|_{L^{\infty}}=|g(0)|\le C\|g\|_{L^2}
\]
for some $C>0$ depending only on $\Sb$.
\end{proof}

\section{Concluding Remarks}

The symmetrized polydisk is a relatively new domain of 
study.  It exhibits some remarkable geometric phenomena
and has demonstrated interesting new properties.
The higher-dimensional generalization of this idea looks particularly promising, and we hope to explore
this idea in subsequent papers.


\bibliographystyle{alpha}
\bibliography{references}

\begin{thebibliography}{NRSW89}

\bibitem[ALY18]{AglLykYou18}
J.~Agler, Z.~A. Lykova, and N.~J. Young.
\newblock Algebraic and geometric aspects of rational {$\Gamma$}-inner
  functions.
\newblock {\em Adv. in Math.}, 328:133--159, 2018.

\bibitem[AY00]{AglYou00}
J.~Agler and N.~J. Young.
\newblock Operators having the symmetrized bidisc as a spectral set.
\newblock {\em Proc. Edinburgh Math. Soc.}, 43(2):195--210, 2000.

\bibitem[AY04]{AglYou04}
J.~Agler and N.~J. Young.
\newblock The hyperbolic geometry of the symmetrized bidisc.
\newblock {\em J. of Geom. Anal.}, 14(3):375--403, 2004.

\bibitem[Bar84]{Bar84}
D.~E. Barrett.
\newblock Irregularity of the {B}ergman projection on a smooth bounded domain
  in {$\C^2$}.
\newblock {\em Ann. of Math.}, 119:431--436, 1984.

\bibitem[Bar92]{Bar92}
D.~E. Barrett.
\newblock Behavior of the {B}ergman projection on the {D}iederich-{F}ornaess
  worm.
\newblock {\em Acta Math.}, 168, 1992.

\bibitem[Bel81]{Bell81}
S.~R. Bell.
\newblock Proper holomorphic mappings and the bergman projection.
\newblock {\em Duke Math. J.}, 48:167--175, 1981.

\bibitem[Bel82]{Bell82c}
S.~R. Bell.
\newblock The {B}ergman kernel function and proper holomorphic mappings.
\newblock {\em Trans. Amer. Math. Soc.}, 270(2):685--691, 1982.

\bibitem[CD06]{CharpentierDupain06}
P.~Charpentier and Y.~Dupain.
\newblock Estimates for the {B}ergman and {S}zeg\"{o} projections for
  pseudoconvex domains of finite type with locally diagonalizable {L}evi forms.
\newblock {\em Publ. Mat.}, 50:413--446, 2006.

\bibitem[Che17]{Chen14}
L.~Chen.
\newblock The {$L^p$} boundedness of the {B}ergman projection for a class of
  bounded {H}artogs domains.
\newblock {\em J. Math. Anal. Appl.}, 448(1):598--610, 2017.

\bibitem[CZ16a]{ChaZey16}
D.~Chakrabarti and Y.~Zeytuncu.
\newblock ${L}^p$ mapping properties of the {B}ergman projection on the
  {H}artogs triangle.
\newblock {\em Proc. Amer. Math. Soc.}, 144(4):1643--1653, 2016.

\bibitem[CZ16b]{ChenZey}
L.~Chen and Y.~E. Zeytuncu.
\newblock Weighted {B}ergman projections on the {H}artogs triangle: exponential
  decay.
\newblock {\em New York J. Math.}, 22:1271--1282, 2016.

\bibitem[CZ18]{ChenZey18}
L.~Chen and Y.~Zeytuncu.
\newblock Smoothing properties of the {F}riedrichs operator on ${L}^p$ spaces.
\newblock {\em Int. J. Math.}, 29(1):1850004 (16 pages), 2018.

\bibitem[EM16]{EdhMcN16}
L.~D. Edholm and J.~D. McNeal.
\newblock The {B}ergman projection on fat {H}artogs triangles: ${L}^p$
  boundedness.
\newblock {\em Proc. Amer. Math. Soc.}, 144(5):2185--2196, 2016.

\bibitem[EM17]{EdhMcN17}
L.~D. Edholm and J.~D. McNeal.
\newblock {B}ergman subspaces and subkernels: degenerate ${L}^p$ mapping and
  zeroes.
\newblock {\em J. Geom. Anal.}, 2017.

\bibitem[Fri37]{Friedrichs}
Kurt Friedrichs.
\newblock On certain inequalities and characteristic value problems for
  analytic functions and for functions of two variables.
\newblock {\em Trans. Amer. Math. Soc.}, 41(3):321--364, 1937.

\bibitem[HM12]{HerMcN10}
A.-K. Herbig and J.~D. McNeal.
\newblock A smoothing property of the {B}ergman projection.
\newblock {\em Math. Ann.}, 354(2):427--449, 2012.

\bibitem[HMS13]{HerMcNStr}
A.-K. Herbig, J.~D. McNeal, and E.~J. Straube.
\newblock Duality of holomorphic function spaces and smoothing properties of
  the bergman projection.
\newblock {\em Trans. Amer. Math. Soc.}, 366(2):647--665, 2013.

\bibitem[Huo18]{Huo18}
Z.~Huo.
\newblock ${L}^p$ estimates for the {B}ergman projection on some {R}einhardt
  domains.
\newblock {\em Proc. Amer. Math. Soc.}, 146(6):2541--2553, 2018.

\bibitem[KP08]{KraPel}
S.~G. Krantz and M.~Peloso.
\newblock The {B}ergman kernel and projection on non-smooth worm domains.
\newblock {\em Houston J. Math.}, 34:873--950, 2008.

\bibitem[LS04]{LanSte04}
L.~Lanzani and E.~M. Stein.
\newblock Szeg\"o and {B}ergman projections on non-smooth planar domains.
\newblock {\em J. Geom. Anal.}, 14(1):63--86, 2004.

\bibitem[MS94]{McNSte94}
J.~D. McNeal and E.~M. Stein.
\newblock Mapping properties of the {B}ergman projection on convex domains of
  finite type.
\newblock {\em Duke Math. J.}, 73(1):177--199, 1994.

\bibitem[NRSW89]{NagRosSteWai89}
A.~Nagel, J.-P. Rosay, E.~M. Stein, and S.~Wainger.
\newblock Estimates for the {B}ergman and {S}zeg{\H o} kernels in {$\C^2$}.
\newblock {\em Ann. of Math. (2)}, 129(1):113--149, 1989.

\bibitem[PS77]{PhoSte77}
D.~H. Phong and E.~M. Stein.
\newblock Estimates for the {B}ergman and {S}zeg\"o projections on strongly
  pseudo-convex domains.
\newblock {\em Duke Math. J.}, 44(3):695--704, 1977.

\bibitem[PS00]{PutinarShapiro1}
M.~Putinar and H.~S. Shapiro.
\newblock The {F}riedrichs operator of a planar domain.
\newblock In {\em Complex analysis, operators, and related topics}, volume 113
  of {\em Oper. Theory Adv. Appl.}, pages 303--330. Birkh\"auser, Basel, 2000.

\bibitem[PS01]{PutinarShapiro2}
M.~Putinar and H.~S. Shapiro.
\newblock The {F}riedrichs operator of a planar domain. {II}.
\newblock In {\em Recent advances in operator theory and related topics
  ({S}zeged, 1999)}, volume 127 of {\em Oper. Theory Adv. Appl.}, pages
  519--551. Birkh\"auser, Basel, 2001.

\bibitem[RZ16]{RavZey}
S.~Ravisankar and Y.~E. Zeytuncu.
\newblock A note on smoothing properties of the {B}ergman projection.
\newblock {\em Internat. J. Math.}, 27(11):1650087, 10, 2016.

\bibitem[Sha87]{ShapiroUnbddQuad}
Harold~S. Shapiro.
\newblock Unbounded quadrature domains.
\newblock In {\em Complex analysis, {I} ({C}ollege {P}ark, {M}d., 1985--86)},
  volume 1275 of {\em Lecture Notes in Math.}, pages 287--331. Springer,
  Berlin, 1987.

\bibitem[Sha92]{ShapiroBook}
Harold~S. Shapiro.
\newblock {\em The {S}chwarz function and its generalization to higher
  dimensions}, volume~9 of {\em University of Arkansas Lecture Notes in the
  Mathematical Sciences}.
\newblock John Wiley \& Sons, Inc., New York, 1992.
\newblock A Wiley-Interscience Publication.

\bibitem[Zey13]{Zey13}
Y.~E. Zeytuncu.
\newblock {$L^p$} regularity of weighted {B}ergman projections.
\newblock {\em Trans. Amer. Math. Soc.}, 365(6):2959--2976, 2013.

\end{thebibliography}

\end{document}